\documentclass[12pt]{amsart}
\usepackage{amsfonts, amssymb, amsmath, amsthm, euscript}
\swapnumbers
\theoremstyle{plain}
\newtheorem{thm}{Theorem}[section]
\newtheorem{lem}[thm]{Lemma}
\newtheorem{prop}[thm]{Proposition}
\newtheorem{cor}[thm]{Corollary}

\theoremstyle{definition}

\newtheorem*{Ack}{Acknowledgement}

\theoremstyle{remark}
\newtheorem{rem}[thm]{Remark}
\newtheorem{note}[thm]{}

\def\rank{\operatorname{rank}}
\def\gr{\operatorname{gr}}
\def\Z{\mathbb{Z}}
\def\Q{\mathbb{Q}}

\begin{document}

\title[Goldie Ranks of Skew Power Series Rings]{Goldie Ranks of Skew Power Series \\ Rings of
  Automorphic Type}

\author{Edward S. Letzter}

\author{Linhong Wang}

\address{Department of Mathematics\\
  Temple University\\
  Philadelphia, PA 19122}

\email{letzter@temple.edu}

\address{Department of Mathematics\\
  Southeastern Louisiana University\\
  Hammond, LA 70402}

\email{lwang@selu.edu}

      \thanks{Research of the first author supported in part by grants
        from the National Security Agency.}



\begin{abstract} Let $A$ be a semprime, right noetherian ring equipped
  with an automorphism $\alpha$, and let $B := A[[y;\alpha]]$ denote
  the corresponding skew power series ring (which is also semiprime
  and right noetherian). We prove that the Goldie ranks of $A$ and $B$
  are equal.  We also record applications to induced ideals.
\end{abstract}

\maketitle


\section{Introduction}

Our primary aim in this note is to show that a skew power series ring of
automorphic type has the same Goldie rank as its coefficient ring,
assuming that the coefficient ring is right or left noetherian and
semiprime.

\begin{note} Studies of Goldie rank in skew polynomial extensions
  include \cite{BelGoo}, \cite{GooLet}, \cite{Grz}, \cite{LerMat},
  \cite{Mat}, \cite{Mus}, \cite{Qui}, \cite{Sho}, and \cite{Sig}. The
  chronology relevant to our present purposes can be briefly
  summarized as follows: (1) In 1972, Shock proved that if $R$ is a
  ring having finite right Goldie dimension, then the right Goldie
  dimension of $R[x]$ is equal to the right Goldie dimension of $R$
  \cite{Sho}. (2) In 1988, Grzeszczuk proved that if $R$ is a
  semiprime right Goldie ring equipped with a derivation $\delta$,
  then the Goldie rank of $R[x;\delta]$ is equal to the Goldie rank of
  $R$ \cite{Grz}. (3) In 1995, Matczuk proved that if $R$ is a
  semiprime right Goldie ring equipped with an automorphism $\tau$ and
  $\tau$-derivation $\delta$, then the Goldie rank of
  $R[x;\tau,\delta]$ is equal to the Goldie rank of $R$ \cite{Mat}. In
  2005, Leroy and Matczuk generalized this last result to the case
  where $\tau$ is an injective endomorphism \cite{LerMat}.
\end{note}

\begin{note}
  Suitably defined skew power series rings $R[[x;\tau,\delta]]$ were
  recently introduced by Venjakob \cite{Ven}, and subsequently studied
  in \cite{SchVen1}, \cite{SchVen2}, and \cite{Wan}. Recent studies of
  skew power series rings $R[[x;\tau]]$, of \emph{automorphic type\/}
  (i.e., with zero skew derivation), include \cite{LetWan},
  \cite{Tug1}, and \cite{Tug2}. Our aim in this note is to initiate
  the study of Goldie ranks of skew power series rings -- beginning
  with extensions of automorphic type.
\end{note}

\begin{note}
  So let $A$ be a right or left noetherian ring equipped with an
  automorphism $\alpha$, and let $B := A[[y;\alpha]]$ denote the
  corresponding skew power series ring. It follows from well-known
  filtered-graded arguments that $B$ is also right or left noetherian
  (see, e.g., \cite[pp.~60--61]{LiVOy}). Further, if $A$ is semiprime
  then $B$ is semiprime; see (\ref{semiprime}). Assuming $A$ is
  semiprime, we show in our main result (\ref{rankA=rankB}) that the
  Goldie ranks of $A$ and $B$ are equal.  The analogous equality also holds for
  the skew Laurent series ring $B' := A[[y^{\pm 1};\alpha]]$.
\end{note}

\begin{note} We also consider corollaries of the above analysis for
  extensions and contractions of ideals between $A$, $B$, and
  $B'$. For example, in (\ref{induced_app}), we see when $A$ is
  noetherian and $I$ is a semiprime $\alpha$-ideal of $A$, that $IB$ is
  a semiprime ideal of $B$ for which $\rank(B/IB) = \rank(A/I)$. (Here and below, ``rank'' by itself refers to Goldie rank.)
\end{note}

\begin{Ack} Most of the material in this note originally formed a part
  (since excised) of our paper \cite{LetWan}.  We thank the original
  referee of \cite{LetWan} for remarks that helped us clarify the
  exposition. We also thank the referees of this paper for
  pointing out significant simplifications in our approach, most
  notably in (\ref{uniserial}) and (\ref{semiprime}).
\end{Ack}

\section{Goldie Rank Equalities}

Throughout, let $A$ be a ring equipped with an automorphism $\alpha$,
and let $B := A[[y;\alpha]]$ denote the ring of skew power series
  \[ \sum_{i = 0}^\infty a_i y^i  =  a_0 + a_1y + a_2y^2 + \cdots,\]
  for $a_0,a_1,\ldots \in A$, and with multiplication determined by
  $ya = \alpha(a)y$ for all $a \in A$. Since $\alpha$ is an
  automorphism, we can just as well write the coefficients on the
  right. (By ``ring'' we will always mean ``associative unital
  algebra.''  Also, all ring homomorphisms and all modules mentioned
  will be assumed to be unital.)

  \begin{note} \label{<y>} As either a left or right $A$-module, we can
  view $B=A[[y;\alpha]]$ as a direct product of copies of $A$, indexed
  by $\{0,1,2,\ldots\}$.  Note that $y$ is normal in $B$ (i.e., $By =
  yB$), and $B/\langle y \rangle$ is naturally isomorphic to $A$.  The $\langle y \rangle$-adic filtration on $B$ is exhaustive, separated, and complete. The associated graded ring $\gr B$ is isomorphic to $A[z;\alpha]$, and so $B$ is right (resp.~left) noetherian, if $A$ is right (resp.~left) noetherian; see, for example, \cite[pp.~60--61]{LiVOy} for details.

  We use $B':=A[[y^{\pm 1};\alpha]]$, the \emph{skew Laurent series ring\/}, to denote the localization of $B$ at powers of $y$. If $B$ is prime (resp.~semiprime), then it is easy to check that $B'$ is prime (resp.~semiprime), noting that nonzero ideals of $B'$ contract to nonzero ideals $B$.
\end{note}

\begin{note} \label{alpha-prime} We will refer to ideals $I$ of $A$ such that $\alpha(I) = I$ as \emph{$\alpha$-ideals}. Next, an $\alpha$-ideal $I$ of $A$ (other than $A$ itself) is \emph{$\alpha$-prime\/} if whenever the product of two $\alpha$-ideals of $A$ is contained in $I$, one of these $\alpha$-ideals must itself be contained in $I$.

If $A$ is right or left noetherian, then an ideal $I$ of $A$ is an $\alpha$-ideal if and only if $\alpha(I) \subseteq I$. Furthermore, if $A$ is right or left noetherian, then it follows from \cite[Remarks $4^*$, $5^*$, p$.$ 338]{GolMic} that the $\alpha$-prime ideals are exactly the intersections $I_1\cap \cdots \cap I_t$ for finite $\alpha$-orbits $I_1,\ldots,I_t$ of prime ideals of $A$.
\end{note}

\begin{note} \label{extend_alpha}
It is not hard to check that $\alpha$ extends to automorphisms of $B$ and $B'$ such that $\alpha(y) = y$. Consequently, we can discuss $\alpha$-ideals and $\alpha$-prime ideals of $B$ and $B'$. If $I$ is an ideal of $B'$, then $\alpha(I) = yIy^{-1} \subseteq I$ and $\alpha^{-1}(I) = y^{-1}Iy \subseteq I$. Therefore, every ideal of $B'$ is an $\alpha$-ideal.
\end{note}

\begin{note}\label{coeffient-ideal}
Let $I$ be an $\alpha$-ideal of $A$. Let $I[[y;\alpha]]$ be the set of power series in $B$ whose coefficients, when written on the left (or the right), are all contained in $I$. Then $I[[y;\, \alpha]]$ is an $\alpha$-ideal of $B$. Let $J$ also be an $\alpha$-ideal of $A$. Then it is not hard to see that
\[I[[y;\, \alpha]]\cdot J[[y;\, \alpha]]\subseteq (IJ)[[y;\, \alpha]],\]
since
\[ \sum_i a_iy^i \cdot \sum_j b_jy^j=\sum_k \Big(\sum_{i+j=k} a_i \cdot \alpha^i(b_j)\Big)y^k,\]
where $a_i\in I,\, b_j\in J$.

An analogous statement holds true for $B'$. That is,
\[I[[y^{\pm};\, \alpha]]\cdot J[[y^{\pm};\, \alpha]]\subseteq (IJ)[[y^{\pm};\, \alpha]],\]
if we continue to assume that $I$ and $J$ are $\alpha$-ideals of $A$.
\end{note}

\begin{prop}\label{AalphaprimeB} The following are equivalent: {\rm (i)} $A$ is $\alpha$-prime, {\rm (ii)} $B$ is $\alpha$-prime, {\rm (iii)} $B'$ is $\alpha$-prime, {\rm (iv)} $B'$ is prime. If we further assume that $A$ is right or left noetherian, then statements {\rm (i)--(iv)} are equivalent to: {\rm (v)} $B$ is prime.
\end{prop}

\begin{proof}
(i)$\Rightarrow$(ii): Assume that $A$ is $\alpha$-prime. To prove that $B$ is $\alpha$-prime, consider two arbitrary elements $f = f_0 + f_1y + f_2y^2 + \cdots$ and $g = g_0 + g_1y + g_2y^2 + \cdots$ of $B$, where $f_0,\,f_1,\,\ldots \in A$ and $g_0,\,g_1,\,\ldots \in A$. Further assume that $fB\alpha^i(g) = 0$ for all integers $i$. It suffices to prove that $f$ or $g$ must equal zero. Now let $m$ be the smallest positive integer such that $f_m \ne 0$, and let $n$ be the smallest positive integer such that $g_n \ne 0$. It follows that $f_m A \alpha^j(g_n) = 0$, for all integers $j$, because $fA\alpha^i(g) = 0$ for all integers $i$. Since $A$ is $\alpha$-prime, it now follows that $f_m$ or $g_n$ is zero, a contradiction. Hence $B$ is $\alpha$-prime.

(iii)$\Leftrightarrow$(iv): Follows from  (\ref{extend_alpha}), where we noted that every ideal of $B'$ is an $\alpha$-ideal.

(ii)$\Rightarrow$(iii): Follows because nonzero $\alpha$-ideals (i.e., ideals) of $B'$
contract to nonzero $\alpha$-ideals of $B$.

(iii)$\Rightarrow$(i): Follows from (\ref{coeffient-ideal}).

Finally, assuming that $A$ is right or left noetherian, the equivalence of (iv) and (v) follows from \cite[10.18]{GooWar}. The proposition follows.
\end{proof}

The following is a straightforward adaptation of the case of commutative power series over a field.

\begin{lem}\label{uniserial}
Let $V$ be a simple right $A$-module. Then $VB$ is a uniserial right $B$-module, and
\[VB\supset VBy \supset VBy^2 \supset \cdots\]
are the only right $B$-submodules of $VB$.
\end{lem}
\begin{proof}
To start, suppose $f=v_0+v_1y+v_2y^2+\cdots$ is a power series in $VB$ with $0\neq v_0\in V$. Note that $v_0 A=V$, since $V$ is simple. Hence $v_1=v_0 a_1$, $v_2=v_0 a_2$, \ldots , for some $a_1,\, a_2,\, \ldots$ in $A$. Let $g=1+u_1 y + u_2 y^2 +\cdots$ be an element in $B$. Then
\begin{align*}
fg&=v_0 + [v_0 u_1 + v_1]y + [v_0u_2+v_1\alpha(u_1)+v_2]y^2 + \cdots\\
&=v_0 + [v_0 u_1 + v_0a_1]y + [v_0u_2+v_0a_1\alpha(u_1)+v_0a_2]y^2 + \cdots.
\end{align*}
Letting $u_1=-a_1,\, u_2=-[a_1\alpha(u_1)+a_2], \cdots$, we have $fg=v_0$. However, $v_0 B=VB$, since $v_0 A=V$. Therefore $fB=VB$. Consequently, $VBy$ is the unique maximal right $B$-submodule of $VB$. Similarly, for all positive integers $i$, $VBy^{i+1}$ is the unique maximal right $B$-submodule of $VBy^i$. The lemma follows.
\end{proof}

\begin{prop} \label{semiprime} {\rm (i)} If $A$ is semiprime, then $B'$ is semiprime. {\rm (ii)} Assume that $A$ is right or left noetherian. Then $A$ is semiprime if and only if $B$ (or $B'$) is semiprime.
\end{prop}

\begin{proof} (i) Let $f = ay^i + a_{i+1}y^{i+1} + \cdots$ be a nonzero Laurent series in $B'$, with initial nonzero term of degree $i \in \Z$, and with initial coefficient $a$. Suppose $fB'f = 0$. Then $f(y^{-i}A)f = 0$, and so $aAa = 0$, a contradiction since $A$ is semiprime. Thus $fB'f \ne 0$ for all nonzero $f \in B'$, and so $B'$ is semiprime.

(ii) Assume that $A$ is right or left noetherian. If $A$ is semiprime, it then follows, for example, from \cite[10.18]{GooWar}, and from (i) that $B$ is semiprime. To show the other direction, let $N$ be the prime radical of $A$. It is clear that $\alpha(N)=N$. Since $A$ is right or left noetherian, $N$ is a nilpotent ideal of $A$. In view of (\ref{coeffient-ideal}), we have $N[[y;\alpha]]$ (or $N[[y^{\pm};\alpha]]$) is a nilpotent ideal in $B$ (or $B'$, respectively). Note that semiprime rings contain no nonzero nilpotent ideals. Therefore, if $B$ or $B'$ is semiprime then $A$ is semiprime.
\end{proof}

\begin{thm} \label{rankA=rankB} Assume that $A$ is semiprime and right or left noetherian. Then $\rank A = \rank B = \rank B'$.
\end{thm}

\begin{proof} We know that $B$ is semiprime, by (\ref{semiprime}), and noetherian, by  (\ref{<y>}). Therefore, since $B'$ is an Ore localization of $B$ at a set of regular elements, it follows that $\rank B = \rank B'$. So it suffices to prove that $\rank A = \rank B$.

Next, it is a straightforward exercise, using the Joseph-Small-Borho-Warfield Additivity Principle, to prove that the Goldie rank of a semiprime right (or left) noetherian ring cannot be less than the Goldie rank of any of its semiprime right (or left) noetherian subrings; see, e.g., \cite[Theorem 1]{War}. So let $E$ be the corresponding right or left Goldie quotient ring of $A$; then $E$ is a semisimple artinian ring. Of course, $\rank A = \rank E$. Also, $\alpha$ extends to $E$, and $B$ embeds as a ring into $E[[y;\alpha]]$. By (\ref{semiprime}), $E[[y;\alpha]]$ is semiprime. It follows that
\[\rank A \leq \rank B \leq \rank E[[y;\alpha]].\]
(Note that $E[[y;\alpha]]$ is not an Ore localization of $B=A[[y;\alpha]]$. For example, $\Q[[x]]$ is not an Ore localization of $\Z[[x]]$.)

We now assume without loss of generality that $A =E$. Next, set
\[ A = V_1 \oplus \cdots \oplus V_d, \]
for simple right ideals $V_1,\ldots,V_d$ of $A$, with $d = \rank A$.
Then
\[ B = V_1B \oplus V_2B \oplus \cdots \oplus V_d B .\]
Therefore, to prove the theorem it suffices to prove that $VB$ is uniform as a right $B$-module if $V$ is an arbitrary simple right ideal of $A$. This follows from (\ref{uniserial}).
\end{proof}

\begin{rem} Remove the assumptions that $A$ is noetherian and
  semiprime. If $\oplus K_i$ is any direct sum of nonzero right ideals
  of $A$, then $\oplus K_iB$ is a direct sum of nonzero right ideals
  of $B$. It therefore follows in full generality that that the right
  (and by symmetry, left) uniform dimension of $B$ cannot be less than
  that of $A$.
\end{rem}

Next, we apply our analysis to induced ideals, again adapting some elementary aspects of the
commutative theory.

\begin{prop} \label{induced} Let $I$ be an ideal of $A$ finitely
  generated on the right and left, and assume further that $\alpha(I)
  = I$. Then $BI = IB$, and (abusing the notation slightly),
\[ B/IB \cong (A/I)[[y;\alpha]] . \]
These statements hold true if $B$ is replaced by $B'$.
\end{prop}

\begin{proof}
Assume first that $I = g_1A + \cdots + g_nA$, for some
$g_1,\ldots,g_n \in I$. Choose
\[ \sum_{i=0}^\infty h_i y^i\; \in \; I[[y;\alpha]],\]
with $h_0,h_1,\ldots \in I$. Then, for suitable choices of $r_{ij}
\in A$,
\[
 \sum_{i=0}^\infty h_iy^i = \sum_{i=0}^\infty (g_1r_{1i}
+ \cdots + g_nr_{ni})y^i = g_1\left(\sum_{i=0}^\infty
r_{1i}y^i\right) + \cdots + g_n\left(\sum_{i=0}^\infty
r_{ni}y^i\right) \in IB.
\]
Hence $I[[y;\alpha]] \subseteq IB$. It is easy to see that $IB
\subseteq I[[y;\alpha]]$, and so $I[[y;\alpha]] = IB$.

Since $\alpha$ is an automorphism and $\alpha(I) = I$, we can also
view $I[[y;\alpha]]$ as the set of power series in $B$ whose
coefficients, when written on the right, are all contained in $I$. A
mirror-image argument to the preceding one now shows that
$I[[y;\alpha]] =BI$. In particular, $I[[y;\alpha]] = IB = BI$ is an
ideal of $B$.  Also, it is easy to see that
$A[[y;\alpha]]/I[[y;\alpha]] \cong (A/I)[[y;\alpha]]$.

Replacing power series with Laurent series, we see that all of the above follows similarly for $B'$. The proposition follows.
\end{proof}

\begin{prop}\label{BalphaprimeAcontract} Assume that $A$ is noetherian and that $K$ is either an $\alpha$-prime ideal of $B$ or a prime ideal of $B'$. Then $K\cap A$ is an $\alpha$-prime ideal of $A$.
\end{prop}

\begin{proof} Assume first that $K$ is an $\alpha$-prime ideal of $B$.
Note that $K\cap A$ is an $\alpha$-ideal of $A$, and let $I$ and $J$ be $\alpha$-ideals of $A$ such that $IJ\subseteq K\cap A$. Then $IJB \subseteq (K\cap A)B\subseteq K$. By (\ref{induced}), $BI = IB$ and $BJ = JB$ are ideals of $B$, and $(IB)(JB) = IJB \subseteq K$. Hence $IB$ or $JB\subseteq K$, since $K$ is $\alpha$-prime. In particular, $I$ or $J\subseteq K$. Hence, $I$ or $J\subseteq K\cap A$. Therefore, $K\cap A$ is an $\alpha$-prime ideal of $A$. The same argument works when $K$ is a prime ideal of $B'$.
\end{proof}

Combining (\ref{AalphaprimeB}),  (\ref{semiprime}), (\ref{rankA=rankB}), (\ref{induced}), and (\ref{BalphaprimeAcontract}) we obtain:

\begin{cor} \label{induced_app} Assume that $A$ is noetherian, and
  recall $B = A[[y;\alpha]]$. Let $I$ be an $\alpha$-ideal of
  $A$. {\rm (i)} $I$ is $\alpha$-prime if and only if $IB$ is
  $\alpha$-prime. {\rm (ii)} $I$ is semiprime if and only if $IB$ is
  semiprime.  {\rm (iii)} If $I$ is semiprime then $\rank B/IB = \rank
  A/I$.
{\rm (iv)} Statements {\rm (i)} through {\rm (iii)} remain
  true if $B$ is replaced by $B' = A[[y^{\pm 1};\alpha]]$.
\end{cor}



\begin{thebibliography}{99}

\bibitem{BelGoo} A. D. Bell and K. R. Goodearl, Uniform rank over
  differential operator rings and Poincar\'e-Birkhoff-Witt
  extensions, \emph{Pacific J. Math.}, 131 (1988), 13-37.

\bibitem{GolMic} A. Goldie and G. Michler, Ore extensions and polycyclic
  group rings, \emph{J. London Math.~Soc.~(2)}, 9 (1974/75), 337--345.

\bibitem{GooLet} K. R. Goodearl and E. S. Letzter, Prime factor
  algebras of the coordinate ring of quantum matrices,
  \emph{Proc.~Amer.~Math.~Soc.}, 121 (1994), 1017--1025.

\bibitem{GooWar} K. R. Goodearl and R. B. Warfield, Jr., \emph{An
      Introduction to Noncommutative Noetherian Rings, Second Edition}, London
    Mathematical Society Student Texts 61, Cambridge University Press,
    Cambridge, 2004.

\bibitem{Grz} P. Grzeszczuk, Goldie dimension of differential
    operator rings, \emph{Comm. Algebra}, 16 (1988), 689-701.

\bibitem{LerMat} A. Leroy and J. Matczuk, Goldie conditions for Ore
    extensions over semiprime rings, \emph{Algebr.~Represent.~Theory},
    8 (2005), 679--688.

\bibitem{LetWan} E. S. Letzter and L. Wang, Prime ideals of
    $q$-commutative power series rings, to appear \emph{Algebr.~Represent.~Theory}.

\bibitem{LiVOy} H. Li and F. Van Oystaeyen, \emph{Zariskian Filtrations}, $K$-Monographs in Mathematics 2, Kluwer Academic Publishers, Dordrecht, 1996.

\bibitem{Mat} J. Matczuk, Goldie rank of Ore extensions,
  \emph{Comm.~Algebra}, 23 (1995), 1455--1471.

\bibitem{Mus} V. A. Mushrub, On the Goldie dimension of Ore
  extensions with several variables, \emph{Fundam.~Prikl.~Mat.}, 7
  (2001), 1107--1121.

\bibitem{Qui} D. Quinn, Embeddings of differential operator rings and
  Goldie dimension, \emph{Proc.~Amer.~Math.~Soc.}, 102 (1988), 9--16.

\bibitem{SchVen1} P. Schneider and O. Venjakob, Localisations and
  completions of skew power series rings, \emph{American Journal of Mathematics}, 132 (2010), 1--36.

\bibitem{SchVen2} \bysame, On the codimension of modules
  over skew power series rings with applications to Iwasawa algebras,
  \emph{J. Pure Appl.~Algebra}, 204 (2006), 349--367.

\bibitem{Sho} R.C. Shock, Polynomial rings over finite dimensional
  rings, \emph{Pacific J. Math.}, 42 (1972), 251-257.

\bibitem{Sig} G. Sigurdsson, Differential operator rings whose prime
  factors have bounded Goldie dimension, \emph{Arch.~Math.~(Basel)},
  42 (1984), 348-353.

\bibitem{Tug1} A. A. Tuganbaev, Laurent series rings and pseudo-differential
  operator rings, \emph{J. Math.~Sci.}, 128 (2005), 2843--2893.

\bibitem{Tug2} \bysame, Skew Laurent series rings and the maximum condition on right annihilators, \emph{Discrete Math.~Appl.}, Vol. 18, (2008), 71--78.

\bibitem{Ven} O. Venjakob, A non-commutative Weierstrass preparation
  theorem and applications to Iwasawa theory, With an appendix by
  Denis Vogel, \emph{J. Reine Angew$.$ Math$.$}, 559 (2003), 153--191.

\bibitem{Wan} L. Wang, Completions of quantum coordinate rings,
  \emph{Proc.~Amer.~Math.~Soc.}, 137 (2009), 911-919.

\bibitem{War} R. B. Warfield, Jr., Prime ideals in ring extensions,
  \emph{J. London Math.~Soc.~(2)}, 28 (1983),453--460.

\end{thebibliography}
\end{document}